\documentclass{article}

\usepackage[T1]{fontenc}
\usepackage{times}
\usepackage{epsfig,graphicx,tabularx,color}
\usepackage[cp850]{inputenc}
\usepackage[english]{babel}
\usepackage{amsmath,amssymb,amsfonts,amsthm,mathrsfs,comment,mathdots,multirow,tikz}
\usepackage{rotating, longtable, relsize}
\usepackage[f]{esvect}
\usepackage{times}
\usepackage{fancyhdr}
\usepackage{lipsum}
\usepackage{soul}

\usetikzlibrary{calc,3d}

\newtheorem{theorem}{Theorem}[section]

\newtheorem{lemma}[theorem]{Lemma}

\theoremstyle{definition}
\newtheorem{remark}[theorem]{Remark}

\def\XXint#1#2#3{{\setbox0=\hbox{$#1{#2#3}{\int}$}
     \vcenter{\hbox{$#2#3$}}\kern-.5\wd0}}

\def\ud{{\rm\,d}}

\def\N{\mathbb{N}}

\def\R{\mathbb{R}}
\def\Sph{\mathbb{S}}

\def\OO{\mathcal{O}}

\newcommand{\bs}[1]{\boldsymbol{#1}}

\def\pr(#1){\left({#1}\right)}
\def\br[#1]{\left[{#1}\right]}

\def\abs#1{\left|{#1}\right|}
\def\norm#1{\left\|{#1}\right\|}

\def\for{\hbox{ for }}

\setulcolor{green}
\soulregister\cite7
\soulregister\ref7
\soulregister\pageref7

\begin{document}

\title{A rapid and well-conditioned algorithm for the Helmholtz--Hodge decomposition of vector fields on the sphere}

\author{Julien Molina and Richard Mika\"el Slevinsky\thanks{Corresponding author. Email: Richard.Slevinsky@umanitoba.ca}\\[2pt] Department of Mathematics, University of Manitoba, Winnipeg, Canada}

\maketitle

\section{The Helmholtz--Hodge decomposition}

Canonical decompositions of a vector field $\bs{V}$ on the unit sphere $\Sph^2\subset\R^3$ exist in Cartesian and spherical coordinates,
\begin{align}
\bs{V}(\bs{r}) & = V^x\bs{e_x} + V^y\bs{e_y} + V^z\bs{e_z},\label{eq:CartesianDecomposition}\\
& = V^r\bs{e_r} + V^\theta\bs{e_\theta} + V^\varphi\bs{e_\varphi}.\label{eq:SphericalDecomposition}
\end{align}
Such canonical decompositions are useful because their numerical evaluation is usually directly available. On the other hand, another decomposition is more useful in the context of vector calculus.

\begin{theorem}[The Helmholtz--Hodge decomposition~\cite{Bhatia-et-al-19-1386-13}]
Every continuous vector field $\bs{V}\in [C(\Sph^2)]^3$ may be decomposed into radial, spheroidal, and toroidal components as
\begin{equation}\label{eq:HelmholtzHodgeDecomposition}
\bs{V}(\bs{r}) = V^r\bs{e_r} + \nabla V^s + \bs{e_r}\times\nabla V^t.
\end{equation}
\end{theorem}

The Helmholtz--Hodge decomposition (HHD) is unique up to two constants, one each for the average values of $V^s$ and $V^t$. In this paper, we minimize the $L^2(\Sph^2)$ norms of the scalar fields by setting the constants to zero.

The classical algorithm~\cite{Bhatia-et-al-19-1386-13} for the HHD is to apply either the divergence or the (normal component of the) curl and solve the resulting Poisson equations independently. This classical algorithm is conceptually simple, but has a few drawbacks. Firstly, when using spherical harmonics, the classical algorithm is immediately applicable for converting the Cartesian decomposition from Eq.~\eqref{eq:CartesianDecomposition} to Eq.~\eqref{eq:HelmholtzHodgeDecomposition}; however, when working in the tangent space to $\Sph^2$, it is redundant to carry around a three-component vector field that is tacitly assumed to satisfy $\bs{V}\cdot\bs{e_r}=0$. Secondly, due to the coordinate singularity, the spherical decomposition from Eq.~\eqref{eq:SphericalDecomposition} to Eq.~\eqref{eq:HelmholtzHodgeDecomposition} has been discredited as being unable to support a stable algorithm~\cite{Swartztrauber-18-191-81}. Finally, the necessity to take the divergence and curl and the solution of Poisson equations requires differentiability and thereby introduces the possibility of non-optimal error growth, with respect to the truncation degree.

The present contribution addresses all three of the drawbacks of the classical algorithm for the HHD. Since the gradient and curl of a scalar field on $\Sph^2$ are tangential to the surface, this paper addresses the conversion of the two representations
\begin{equation}\label{eq:SpheroidalToroidal}
V^\theta\bs{e_\theta} + V^\varphi\bs{e_\varphi} = \nabla V^s + \bs{e_r}\times\nabla V^t,
\end{equation}
that is, we suppose that the radial component has been resolved independently. The drawbacks are addressed by working with suitable orthonormal bases for the angular components of the vector field in spherical coordinates and by inverting Eq.~\eqref{eq:SpheroidalToroidal} directly, rather than solving Poisson equations. The algorithm uncouples modes of spherical harmonics with different absolute order, writes the inversion of Eq.~\eqref{eq:SpheroidalToroidal} as barely-overdetermined\footnote{Our notion of a barely-overdetermined linear system is a rectangular system with precisely two extra rows, independent of the total dimensions.} banded linear least-squares systems, and solves them with banded $QR$ decompositions that factor and execute in optimal complexity. Rigorous upper bounds on the $2$-norm relative condition number of the banded linear systems demonstrate the low error growth with truncation degree and build confidence in the approach.

Vector spherical harmonics~\cite{Barrera-Estevez-Giraldo-6-287-85} are a vector-valued basis for vector fields in $[L^2(\Sph^2)]^3$ that already conform to the HHD. Thus, synthesis and analysis with vector spherical harmonics are a natural direct solution to the HHD, and several software libraries offer such computational routines, see e.g.~\cite{Schaeffer-14-751-13}. However, a slight modification of the fast and backward stable (scalar) spherical harmonic transforms of Slevinsky~\cite{Slevinsky-ACHA-17,Slevinsky-1711-07866} allow for synthesis and analysis with our to-be-defined auxiliary orthonormal basis. Furthermore, they are preferable due to the lower error growth as a function of the truncation degree.

\section{Spherical harmonics}
The unit sphere $\Sph^2$ may be parameterized by the co-latitudinal angle $\theta\in[0,\pi]$ and the longitudinal angle $\varphi\in[0,2\pi)$. Real spherical harmonics~\cite{Atkinson-Han-12} are the separable eigenfunctions of the Laplace--Beltrami operator on $\Sph^2$, and are given in terms of associated Legendre functions as
\begin{equation}
Y_{\ell,m}(\theta, \varphi) = \tilde{P}_\ell^{\abs{m}}(\cos\theta) \times \sqrt{\frac{2-\delta_{m,0}}{2\pi}}\times \left\{\begin{array}{ccc}
\cos (m\varphi ) & \for & m \ge 0,\\
\sin (-m\varphi ) & \for & m<0,
\end{array}\right.
\end{equation}
where
\begin{equation}
\tilde{P}^{\abs{m}}_\ell(\cos\theta) = (-1)^{\abs{m}}\sqrt{(\ell+\tfrac{1}{2})\frac{(\ell-\abs{m})!}{(\ell+\abs{m})!}}P_\ell^{\abs{m}}(\cos\theta),
\end{equation}
are $L^2$-normalized associated Legendre functions. Altogether, spherical harmonics form a complete orthonormal basis on $L^2(\Sph^2)$ with respect to the Lebesgue measure. Thus, every function $f\in L^2(\Sph^2)$ may be expanded in this basis
\begin{equation}
f(\theta,\varphi) = \sum_{\ell=0}^{+\infty}\sum_{m=-\ell}^{+\ell}f_{\ell,m} Y_{\ell,m}(\theta,\varphi),
\end{equation}
where the expansion coefficients are given by the inner product
\begin{equation}
f_{\ell,m} = \int_{\Sph^2} f(\theta,\varphi) Y_{\ell,m}(\theta,\varphi)\ud\Omega.
\end{equation}

Now, spherical harmonics may be expanded using precisely half of all the tensor-product Fourier modes. The complement of the Fourier modes that support spherical harmonics may be used altogether to define another orthonormal basis, this time useful to represent the angular components of vector fields, $V^\theta$ and $V^\varphi$, in Eq.~\eqref{eq:SphericalDecomposition}. We call this basis $Z_{\ell,m}$, and they are defined as
\begin{equation}
Z_{\ell,m}(\theta,\varphi) := \tilde{P}^{\abs{\abs{m}-1}}_\ell(\cos\theta) \times \sqrt{\frac{2-\delta_{m,0}}{2\pi}}\times \left\{\begin{array}{ccc}
\cos (m\varphi ) & \for & m \ge 0,\\
\sin (-m\varphi ) & \for & m<0.
\end{array}\right.
\end{equation}

It is easy to verify that the surface gradient
\begin{equation}
\nabla^* = \nabla_\theta\bs{e_\theta}+\nabla_\varphi\bs{e_\varphi} := \partial_\theta \bs{e_\theta} + \csc\theta\partial_\varphi \bs{e_\varphi},
\end{equation}
and the surface curl, $\bs{e_r}\times\nabla^*$, of spherical harmonics may be expanded in this scalar basis with vector-valued coefficients. Ultimately,
\begin{equation}
\bs{V}(\bs{r}) = \sum_{m=-\infty}^{+\infty}\sum_{\ell=\abs{\abs{m}-1}}^{+\infty} \left(V_{\ell,m}^\theta\bs{e_\theta}+V_{\ell,m}^\varphi\bs{e_\varphi}\right)Z_{\ell,m}(\theta,\varphi),
\end{equation}
and the algorithm we describe returns the expansion coefficients in the spheroidal--toroidal decomposition,
\begin{equation}
\bs{V}(\bs{r}) = \sum_{\ell=0}^{+\infty}\sum_{m=-\ell}^{+\ell} \left(V_{\ell,m}^s \nabla Y_{\ell,m}(\theta,\varphi) + V_{\ell,m}^t \bs{e_r}\times \nabla Y_{\ell,m}(\theta,\varphi)\right).
\end{equation}

\section{A rapid algorithm for the HHD}

In case $m=0$, the surface gradient and surface curl are readily separable, thus we focus on $\abs{m}>0$.

We will be dealing with three types of bases: the vector spherical harmonics $\nabla Y_{\ell,m}$ and $\bs{e_r}\times\nabla Y_{\ell,m}$; the new $Z_{\ell,m}$ basis with which we resolve the vector field; and, an intermediary basis $\csc\theta Y_{\ell,m}$ to express both components of the gradient of a vector field. We sum up this relationship in the following diagram:
\begin{equation}
Z_{\ell,m} \underset{1 \ \text{matrix}}{\longleftrightarrow} \csc \theta Y_{\ell,m} \underset{\text{two matrices}}{\longleftrightarrow} \left\{\begin{array}{c} \nabla_\theta Y_{\ell,m}\\ \nabla_\varphi Y_{\ell,m}\end{array}\right..
\end{equation}

\begin{enumerate}
\item First, we examine how to convert expansions in the $Z_{\ell,m}$ basis to expansions in the $\csc \theta Y_{\ell,m}$ basis. Thanks to~\cite[Eq.~(6.6)]{Swartztrauber-18-191-81},
\begin{equation}
Z_{\ell,m} = \alpha_\ell^m\csc \theta Y_{\ell-1,m} +  \beta_\ell^m\csc \theta Y_{\ell+1,m},
\end{equation}
where
\begin{equation}
\alpha_\ell^m = -\sqrt{\frac{(\ell-m)(\ell-m+1)}{(2\ell-1)(2\ell+1)}}\quad{\rm and}\quad \beta_\ell^m = \displaystyle \sqrt{\frac{(\ell+m)(\ell+m+1)}{(2\ell+1)(2\ell+3)}}.
\end{equation}
Therefore, for a component of a vector field, we obtain
\begin{equation}\label{eq:ZlmtocscthetaYlm}
\sum_{m=-n}^{+n}\sum_{\ell=\abs{\abs{m}-1}}^n V_{\ell,m}Z_{\ell,m} = \sum_{\ell=0}^n\sum_{m=-\ell}^{+\ell}\left( V_{\ell+1,m}\alpha_{\ell+1}^m + V_{\ell-1,m}\beta_{\ell-1}^m \right) \csc\theta Y_{\ell,m}.
\end{equation}

\item Next, we examine how to represent the gradient of spherical harmonics in terms of $\csc\theta Y_{\ell,m}$. This requires two different matrices, one for each component.
\begin{enumerate}
\item For $\nabla_\theta$, thanks to~\cite[Eq.~(6.5)]{Swartztrauber-18-191-81},
\begin{equation}
\nabla_{\theta}Y_{\ell,m} = \partial_{\theta}Y_{\ell,m} = \gamma_\ell^m \csc\theta Y_{\ell-1,m} + \delta_\ell^m \csc\theta Y_{\ell+1,m},
\end{equation}
where
\begin{equation}
\gamma_\ell^m = -(\ell+1)\sqrt{\frac{(\ell-m)(\ell+m)}{(2\ell-1)(2\ell+1)}}\quad{\rm and}\quad\delta_\ell^m = \ell\sqrt{\frac{(\ell-m+1)(\ell+m+1)}{(2\ell+1)(2\ell+3)}}.
\end{equation}
Thus, for the finite expansion
\begin{align}
\sum_{\ell=0}^{n-1}\sum_{m=-\ell}^{+\ell}V_{\ell,m}\nabla_{\theta}Y_{\ell,m} & = \sum_{\ell=0}^{n-1}\sum_{m=-\ell}^{+\ell}V_{\ell,m}(\gamma_\ell^m\csc\theta Y_{\ell-1,m} + \delta_\ell^m\csc\theta Y_{\ell+1,m}),\\
& = \sum_{\ell=0}^n\sum_{m=-\ell}^{+\ell}\left( V_{\ell+1,m}\gamma_{\ell+1}^m + V_{\ell-1,m}\delta_{\ell-1}^m \right) \csc\theta Y_{\ell,m}.
\end{align}
For every $1\le \abs{m}\le n-1$, let $A\in\R^{(n+1-m)\times(n-m)}$ be the matrix that represents $\nabla_\theta$:
\begin{equation}
A = \begin{pmatrix}
0 & \gamma_{\vert m \vert +1}^m\\
\delta_{\vert m \vert}^m & 0 & \gamma_{\vert m \vert +2}^m\\
 & \delta_{\vert m \vert +1}^m & 0 & \gamma_{\vert m \vert +3}^m\\
 & & \delta_{\vert m \vert +2}^m & 0 & \ddots\\
 & & & \ddots & \ddots & \gamma_n^m \\
 & & & & \ddots & 0 \\
 & & & & & \delta_n^m\\
\end{pmatrix}.
\end{equation}
\item For $\nabla_\varphi$, we notice that
\begin{equation}
\nabla_{\varphi}Y_{\ell,m} = \csc\theta \partial_{\varphi}Y_{\ell,m} = -m\csc \theta Y_{\ell,-m}.
\end{equation}
Thus, for the finite expansion
\begin{align}
\sum_{\ell=0}^{n-1}\sum_{m=-\ell}^{+\ell}V_{\ell,m}\nabla_\varphi Y_{\ell,m} & = \sum_{\ell=0}^{n-1}\sum_{m=-\ell}^{+\ell}(-m)V_{\ell,m}\csc\theta Y_{\ell,-m},\\
& = \sum_{\ell=0}^{n-1}\sum_{m=-\ell}^{+\ell} \left(m V_{\ell,-m}\right)\csc\theta Y_{\ell,m}.
\end{align}
For every $1\le \abs{m}\le n-1$, let $B\in\R^{(n+1-m)\times(n-m)}$ be the matrix that represents $\nabla_\varphi$:
\begin{equation}
B = \begin{pmatrix}
m\\
 & m\\
 & & \ddots\\
 & & & m\\
0 & \cdots &  & \cdots & 0\\
\end{pmatrix}.
\end{equation}
\end{enumerate}
\end{enumerate}

For every $1\le \abs{m}\le n-1$, let $\tilde{V}_{\ell,m}^\theta$ and $\tilde{V}_{\ell,m}^\varphi$ be the coefficients of the angular components of the vector field in the $\csc\theta Y_{\ell,m}$ basis, the result of back substitution in Eq.~\eqref{eq:ZlmtocscthetaYlm}. From here, the linear system
\begin{equation}\label{eq:HHDsystem}
\begin{pmatrix} A & B\\B & A\end{pmatrix}\begin{pmatrix} V_{:,m}^s & V_{:,-m}^s\\ V_{:,-m}^t & -V_{:,m}^t\end{pmatrix} = \begin{pmatrix} \tilde{V}_{:,m}^\theta & \tilde{V}_{:,m}^\varphi\\ \tilde{V}_{:,-m}^\theta & \tilde{V}_{:,-m}^\varphi\end{pmatrix},
\end{equation}
summarizes the relationships defined by Eq.~\eqref{eq:SpheroidalToroidal}. We use the shorthand notation $V_{:,m}$ to denote the vector containing all pertinent entries of the particular field expansion of order $m$.

The linear system in Eq.~\eqref{eq:HHDsystem} is barely-overdetermined and sparse since $A$ is tridiagonal and $B$ is diagonal. To capitalize on the sparsity, we employ the perfect shuffle permutations $P_1 = I_{2(n+1-m)}[:,\pi_1]$ and $P_2 = I_{2(n-m)}[:,\pi_2]$, where $\pi_1$ and $\pi_2$ are permutations of conformable sizes that collect the odd numbers before the even numbers. Then, the linear systems $P_1 \begin{pmatrix} A & B\\ B & A\end{pmatrix}P_2^\top$ are pentadiagonal. The permutations themselves are applied rapidly as they amount to an interleaving of the input and the output. We solve the overdetermined linear systems via least-squares, employing a $QR$ factorization that respects the banded structure of the permuted system. For every $\abs{m}$, the solution of the least-squares problem takes $\OO(n)$ time to factorize and solve, resulting in the optimal complexity of $\OO(n^2)$ for the total HHD.

\section{On the condition of the algorithm for the HHD}

Let $M$ be the rectangular linear system that we solve for every $m$ in the Helmholtz--Hodge decomposition,
\begin{equation}
M = \begin{pmatrix} A & B\\ B & A\end{pmatrix},
\end{equation}
We are interested in the conditioning of the linear system in particular in terms of the truncation degree $n$ and the order $m$. The $2$-norm relative condition number may be defined in terms of $M^\top M$ by:
\begin{equation}
\kappa_2(M) := \sqrt{ \norm{M^\top M}_2\norm{(M^\top M)^{-1}}_2}.
\end{equation}
Although $M$ is a block rectangular linear system, $M^\top M$ is square,
\begin{equation}
M^\top M = \begin{pmatrix} A^\top A + B^\top B & A^\top B + B^\top A\\ A^\top B + B^\top A & A^\top A + B^\top B\end{pmatrix} =: \begin{pmatrix} C & D\\ D & C\end{pmatrix},
\end{equation}
where $C$ is symmetric and pentadiagonal with no entries on the first sub- and super-diagonals, and $D$ is symmetric and tridiagonal with no entries on the main diagonal. By an analysis of block determinants, $\lambda(M^\top M) = \lambda(C+D)\cup\lambda(C-D) \equiv \lambda(C+D)$, since $C+D$ and $C-D$ are diagonally similar. Therefore,
\begin{equation}
\kappa_2(M) = \sqrt{\kappa_2(C+D)}.
\end{equation}
Normally, this would suffice for an analysis of the conditioning of $C+D$ by the use of Ger\v sgorin discs. However, such an analysis is frustrated for every disc contains the origin when $m=1$. Instead, we find the Cholesky factorization of $C+D = R^\top R$ directly, as
\begin{equation}
R = \begin{pmatrix} d_1 & -e_1 & -f_1\\ & d_2 & -e_2 & -f_2\\ & & d_3 & -e_3 & -f_3\\ & & & \ddots & \ddots & \ddots\end{pmatrix},
\end{equation}
which is easy to confirm {\em a posteriori},
\begin{align}
d_\ell & = (\ell+m-1)\sqrt{\frac{(\ell+m+1)(\ell+2m)(\ell+2m+1)}{(\ell+m)(2\ell+2m-1)(2\ell+2m+1)}},\\
e_\ell & = \sqrt{\frac{\ell(\ell+2m+1)}{(\ell+m)(\ell+m+1)}},\\
f_\ell & = (\ell+m+2)\sqrt{\frac{\ell(\ell+1)(\ell+m)}{(\ell+m+1)(2\ell+2m+1)(2\ell+2m+3)}}.
\end{align}
With $R$ in hand,
\begin{equation}
\kappa_2(M) = \kappa_2(R) = \norm{R}_2\norm{R^{-1}}_2,
\end{equation}
and with this simplification, estimation of the condition number follows naturally.

The following theorem is proved in Appendix~\ref{appendix:kappa}.

\begin{theorem}\label{theorem:kappa}
For every $n\in\N$, let $R\in\R^{n\times n}$. If $m=1$, then
\begin{equation}
\kappa_2(R) \le \left(n+\tfrac{5}{2}\right)\left(4e^{1+\frac{7\pi^2}{8}}\left[2+\log n\right]\right).
\end{equation}
Otherwise, if $m\ge2$, then
\begin{equation}
\kappa_2(R) \le \frac{n+m+\frac{3}{2}}{m-\frac{3}{2}}.
\end{equation}
\end{theorem}

Figure~\ref{fig:errortime} shows the numerical results illustrating the rapidity of the algorithm for the HHD and its well-conditioning. Coefficients of the spheroidal and toroidal components are drawn from the standard normal distribution, and the components are differentiated and expressed in the $Z_{\ell,m}$ basis, and separated by the HHD. The relative $\ell^2$-norm of the error is depicted, where the well-conditioning described in theorem~\ref{theorem:kappa} is borne out in practice as a statistical error bound of $\OO(\sqrt{\kappa_2(M)}\varepsilon)$. Our implementation of the algorithm is freely available in~\cite{MultivariateOrthogonalPolynomials}.

\begin{figure}[t]
\begin{center}
\begin{tabular}{cc}
\hspace*{-0.2cm}\includegraphics[width=0.53\textwidth]{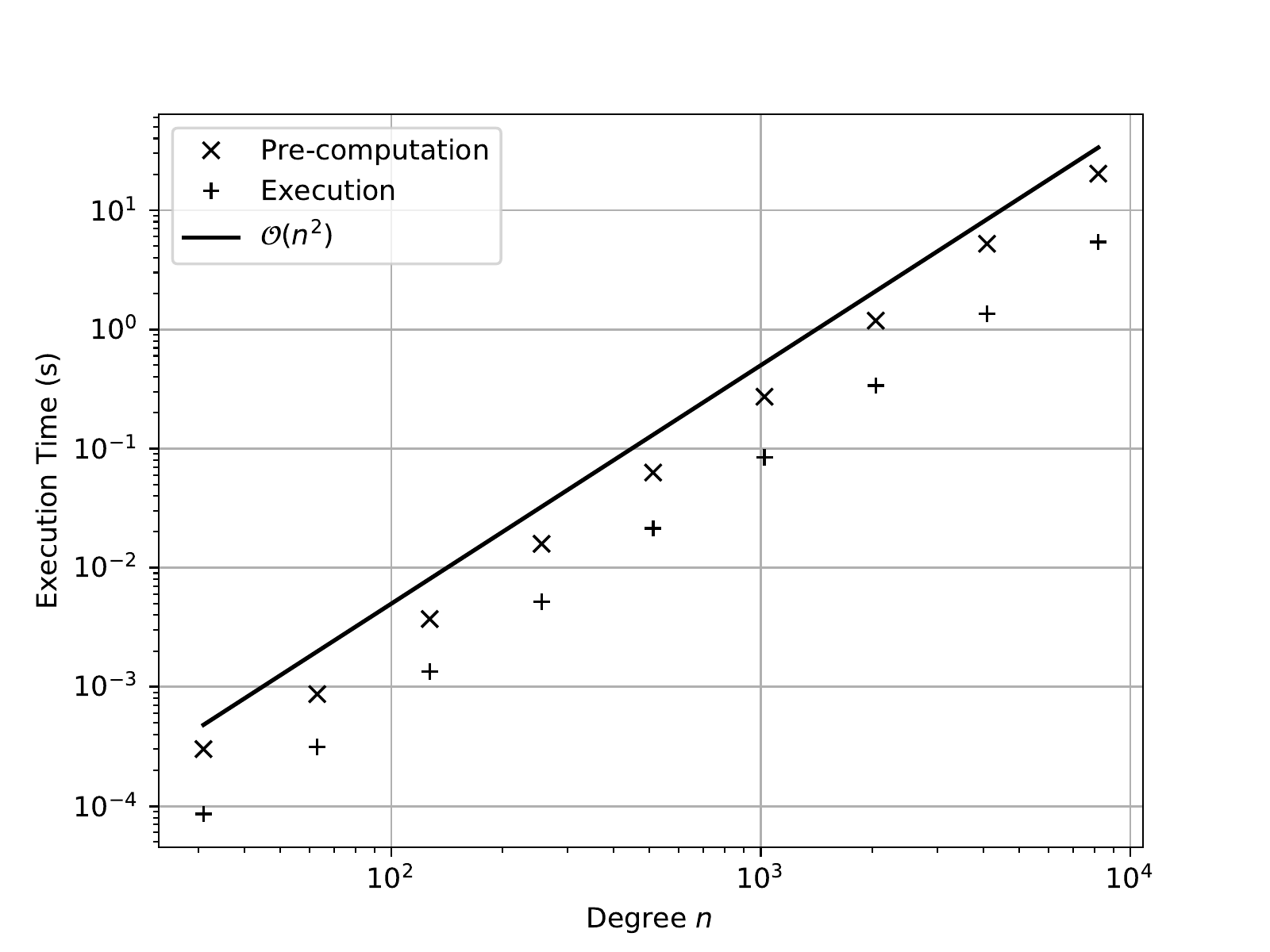}&
\hspace*{-0.65cm}\includegraphics[width=0.53\textwidth]{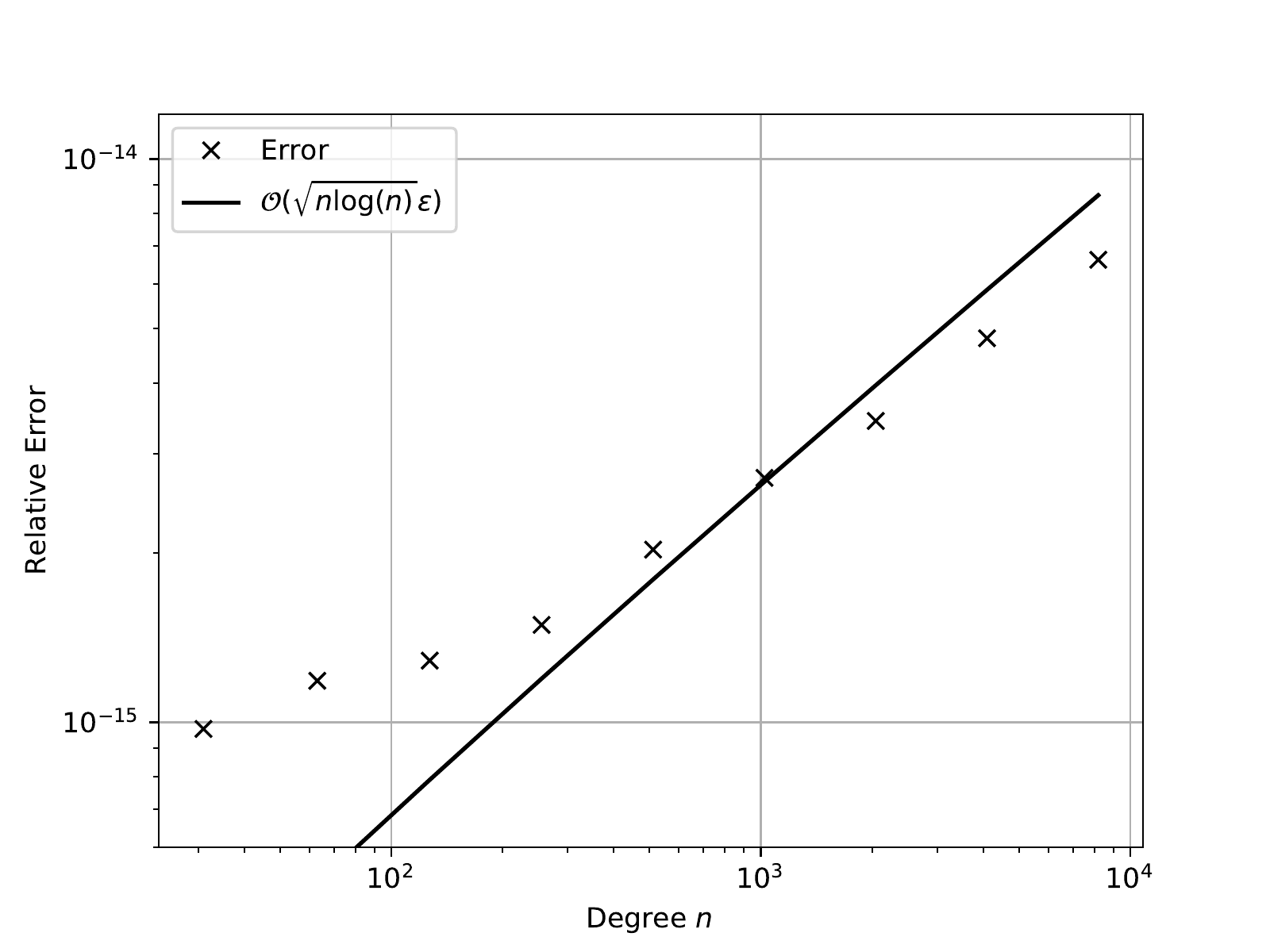}\\
\end{tabular}
\caption{Left: Pre-computation and execution times of the Helmholtz--Hodge decomposition. In both plots, observations are averaged over $10$ iterations to reduce the variance. Right: relative $\ell^2$-norm of the standard normally distributed coefficients of the spheroidal and toroidal components after differentiation and decomposition with $V_{0,0}^s = V_{0,0}^t = 0$.}
\label{fig:errortime}
\end{center}
\end{figure}

\section*{Acknowledgments}

We thank Sheehan Olver for providing a freely available banded $QR$ factorization that respects the bandwidth. We thank the Natural Sciences and Engineering Research Council of Canada (RGPIN-2017-05514), Universit\'e Lyon 1 and la r\'egion Rh\^one--Alpes for providing the financial support for the first author to visit the University of Manitoba for three months.

\bibliography{/Users/Mikael/Bibliography/Mik}

\appendix

\section{Proof of Theorem~\ref{theorem:kappa}}\label{appendix:kappa}

\begin{lemma}\label{lemma:upper}
The following inequalities hold for $\ell\in\N$ and $m\in\N$,
\begin{equation}
d_\ell \le \frac{\ell+2m}{2},\qquad e_\ell \le 1,\quad{\rm and}\quad f_\ell \le \frac{\ell+1}{2}.
\end{equation}
\end{lemma}
\begin{proof}
The inequalities are demonstrated by a careful rewriting of the argument of the square root. For $d_\ell$,
\begin{align}
d_\ell & = (\ell+m-1)\sqrt{\frac{(\ell+m+1)(\ell+2m)(\ell+2m+1)}{(\ell+m)(2\ell+2m-1)(2\ell+2m+1)}},\\
& = (\ell+2m)\sqrt{\frac{(\ell+m-1)^2(\ell+m+1)(\ell+2m+1)}{(\ell+m)(\ell+2m)(2\ell+2m-1)(2\ell+2m+1)}},\\
& = \frac{\ell+2m}{2}\sqrt{1 - \frac{4\ell^2m+8\ell m^2+4m^3+7\ell^2+17\ell m+10m^2-4m-4}{(\ell+m)(\ell+2m)(2\ell+2m-1)(2\ell+2m+1)}}.
\end{align}
For $e_\ell$,
\begin{equation}\label{eq:1superineq}
e_\ell = \sqrt{\frac{\ell(\ell+2m+1)}{(\ell+m)(\ell+m+1)}} = \sqrt{1 - \frac{m(m+1)}{(\ell+m)(\ell+m+1)}}.
\end{equation}
And for $f_\ell$,
\begin{align}
f_\ell & = (\ell+m+2)\sqrt{\frac{\ell(\ell+1)(\ell+m)}{(\ell+m+1)(2\ell+2m+1)(2\ell+2m+3)}},\\
& = (\ell+1)\sqrt{\frac{\ell(\ell+m)(\ell+m+2)^2}{(\ell+1)(\ell+m+1)(2\ell+2m+1)(2\ell+2m+3)}},\\
& = \frac{\ell+1}{2}\sqrt{1-\frac{4\ell^2m+8\ell m^2 + 4m^3+7\ell^2+19\ell m + 12m^2+14\ell+11m+3}{(\ell+1)(\ell+m+1)(2\ell+2m+1)(2\ell+2m+3)}}.\label{eq:2superineq}
\end{align}
\end{proof}

\begin{lemma}\label{lemma:lower}
For $m\ge2$, every row sum of $R$ is uniformly bounded below
\begin{equation}
d_\ell-e_\ell-f_\ell \ge m-\frac{3}{2}.
\end{equation}
\end{lemma}
\begin{proof}
We must refine some estimates from the previous lemma. Using
\begin{equation}
\sqrt{1+x} \le 1+\frac{x}{2},\quad\forall x\ge-1,
\end{equation}
together with Eqs.~\eqref{eq:1superineq} and~\eqref{eq:2superineq}, we find
\begin{align}
e_\ell & \le 1-\frac{m(m+1)}{2(\ell+m)(\ell+m+1)},\\
f_\ell & \le \frac{\ell+1}{2}\left[1-\frac{4\ell^2m+8\ell m^2 + 4m^3+7\ell^2+19\ell m + 12m^2+14\ell+11m+3}{2(\ell+1)(\ell+m+1)(2\ell+2m+1)(2\ell+2m+3)}\right].
\end{align}
Similarly, using
\begin{equation}
\frac{1}{\sqrt{1+x}} \ge 1-\frac{x}{2},\quad\forall x\ge-1,
\end{equation}
with
\begin{align}
d_\ell & = \dfrac{\ell+2m}{\sqrt{\dfrac{(\ell+m)(\ell+2m)(2\ell+2m-1)(2\ell+2m+1)}{(\ell+m-1)^2(\ell+m+1)(\ell+2m+1)}}},\\
& = \dfrac{\ell+2m}{\sqrt{4+\dfrac{4\ell^2m+8\ell m^2+4m^3+7\ell^2+17\ell m + 10m^2-4m-4}{(\ell+m-1)^2(\ell+m+1)(\ell+2m+1)}}},\\
& \ge \frac{\ell+2m}{2}\left[1-\frac{4\ell^2m+8\ell m^2+4m^3+7\ell^2+17\ell m + 10m^2-4m-4}{8(\ell+m-1)^2(\ell+m+1)(\ell+2m+1)}\right].
\end{align}
The row sum is therefore
\begin{align}
d_\ell-e_\ell-f_\ell \ge m-\frac{3}{2} + & \Big[(32m^2-8m-28)\ell^5 + (192m^3-32m^2-232m-77)\ell^4\nonumber\\
& + (448m^4-48m^3-776m^2-426m-8)\ell^3\nonumber\\
& + (512m^5-32m^4-1288m^3-877m^2+28m+56)\ell^2\nonumber\\
& + (288m^6-8m^5-1036m^4-768m^3+136m^2+164m+12)\ell\label{eq:lowerbound}\\
& + (64m^7-320m^5-240m^4+124m^3+156m^2+36m)\Big]\nonumber\\
& \Bigg/ \Big[(\ell+m-1)^2(\ell+m)(\ell+m+1)(\ell+2m+1)(2\ell+2m+1)(2\ell+2m+3)\Big].\nonumber
\end{align}
Every coefficient of the numerator of the rational function on the right-hand side of Eq.~\eqref{eq:lowerbound} expressed as a polynomial in $\ell$ is a Hurwitz polynomial in the variable $m-3$, that is, it is a polynomial with positive coefficients. For $m=2$, the numerator is almost Hurwitz in $\ell$,
\begin{equation}
84\ell^5 + 867\ell^4+2820\ell^3+2172\ell^2-3660\ell-4200.
\end{equation}
In fact, it is Hurwitz in the variable $\ell-2$. To complete the proof, we confirm directly for $\ell=1$ and $m=2$ that
\begin{align}
d_1-e_1-f_1 & = \sqrt{\frac{32}{7}} - \sqrt{\frac{1}{2}} - \sqrt{\frac{25}{42}} = \frac{\sqrt{192}-\sqrt{21}-\sqrt{25}}{\sqrt{42}},\\
& > \frac{\sqrt{192}-10}{\sqrt{42}} = \frac{192-100}{(\sqrt{192}+10)\sqrt{42}},\\
& > \frac{92}{(\sqrt{196}+10)\sqrt{42}} = \frac{23}{6\sqrt{42}},\\
& > \frac{21}{6\sqrt{42}} = \frac{\sqrt{42}}{12} > \frac{\sqrt{36}}{12} = \frac{1}{2}.
\end{align}
\end{proof}

{\em Proof of Theorem~\ref{theorem:kappa}.}
For every $m$, we find an upper bound for the largest singular value of $R$ based on~\cite{Qi-56-105-84},
\begin{align}
\sigma_1 & \le \max\left[\max_{1\le\ell\le n}\{ d_\ell + e_\ell + f_\ell\}, \max_{1\le\ell\le n}\{ d_\ell + e_{\ell-1} + f_{\ell-2}\}\right],\\
& \le \max_{1\le\ell\le n}\{ d_\ell + e_\ell + f_\ell\} = d_n+e_n+f_n \le n+m+\frac{3}{2},
\end{align}
where the last inequality follows from Lemma~\ref{lemma:upper}. Note that if the index $\ell<1$, we set the result to $0$.

For $m\ge2$, we find a lower bound for the smallest singular value of $R$ similarly based on~\cite{Qi-56-105-84},
\begin{align}
\sigma_n & \ge \min\left[\min_{1\le\ell\le n}\{ d_\ell - e_\ell - f_\ell\}, \min_{1\le\ell\le n}\{ d_\ell - e_{\ell-1} - f_{\ell-2}\}\right],\\
& \ge \min_{1\le\ell\le n}\{ d_\ell - e_\ell - f_\ell\} \ge m-\frac{3}{2},
\end{align}
where the last inequality follows from Lemma~\ref{lemma:lower}.

For $m=1$, we find an upper bound on the norm of the inverse based on the Frobenius norm~\cite[(2.3.7)]{Golub-Van-Loan-13},
\begin{equation}
\norm{R^{-1}}_2 \le \norm{R^{-1}}_F.
\end{equation}

We represent $R^{-1}$ via a block semi-separable times block diagonal form~\cite[(12.2.2)]{Golub-Van-Loan-13}. Let
\begin{equation}
R = \begin{pmatrix} a_1 & b_1\\ & a_2 & b_2\\ & & a_3 & b_3\\ & & & \ddots & \ddots\end{pmatrix} = \begin{pmatrix} a_1\\ & a_2\\ & & a_3\\ & & & \ddots\end{pmatrix}\begin{pmatrix} I & -c_1\\ & I & -c_2\\ & & I & -c_3\\ & & & \ddots & \ddots\end{pmatrix},
\end{equation}
where
\begin{equation}
a_\ell = \begin{pmatrix} d_{2\ell-1} & -e_{2\ell-1}\\ & d_{2\ell}\end{pmatrix},\qquad b_\ell = -\begin{pmatrix} f_{2\ell-1}\\ e_{2\ell} & f_{2\ell}\end{pmatrix},
\end{equation}
and where
\begin{equation}
a_\ell^{-1} = \begin{pmatrix} d_{2\ell-1}^{-1} & \frac{e_{2\ell-1}}{d_{2\ell-1}d_{2\ell}}\\ & d_{2\ell}^{-1}\end{pmatrix},\qquad c_\ell = -a_\ell^{-1}b_\ell = \begin{pmatrix} \frac{f_{2\ell-1}d_{2\ell} + e_{2\ell-1}e_{2\ell}}{d_{2\ell-1}d_{2\ell}} & \frac{e_{2\ell-1}f_{2\ell}}{d_{2\ell-1}d_{2\ell}}\\ \frac{e_{2\ell}}{d_{2\ell}} & \frac{f_{2\ell}}{d_{2\ell}}\end{pmatrix}.
\end{equation}
Then~\cite[(12.2.3)]{Golub-Van-Loan-13}
\begin{equation}
R^{-1} = \begin{pmatrix} I & c_1 & c_1c_2 & c_1c_2c_3 & \cdots\\ & I & c_2 & c_2c_3 & \cdots\\ & & I & c_3 & \cdots\\ & & & \ddots & \ddots\end{pmatrix}\begin{pmatrix} a_1^{-1}\\ & a_2^{-1}\\ & & a_3^{-1}\\ & & & \ddots\end{pmatrix}.
\end{equation}
Now, the Frobenius norm of $R^{-1}\in\mathbb{R}^{2n\times 2n}$, say, is given in terms of the sum of the block Frobenius norms
\begin{equation}
\norm{R^{-1}}_F^2 = \sum_{j=1}^n\sum_{i=1}^j \norm{c_i\cdots c_{j-1} a_j^{-1}}_F^2.
\end{equation}
Here, the empty product $c_j\cdots c_{j-1} \equiv I$. For each $C\in\mathbb{R}^{2\times2}$ block, we relate the Frobenius norm to the $\infty$-norm, $\norm{C}_F^2 \le 4\norm{C}_\infty^2$~\cite[(2.3.7) \& (2.3.11)]{Golub-Van-Loan-13}, and by submultiplicativity of $\infty$-norms for the products
\begin{equation}
\norm{R^{-1}}_F^2 \le 4\sum_{j=1}^n\sum_{i=1}^j \norm{c_i}_\infty^2 \cdots \norm{c_{j-1}}_\infty^2 \norm{a_j^{-1}}_\infty^2.
\end{equation}
All we need to do is bound the $\infty$-norms of $a_\ell^{-1}$ and $c_\ell$. For $a_\ell^{-1}$,
\begin{align}
\norm{a_\ell^{-1}}_\infty & = \max\left\{\frac{1}{d_{2\ell-1}}\left(1+\frac{e_{2\ell-1}}{d_{2\ell}}\right), \frac{1}{d_{2\ell}}\right\},\\
& = \frac{1}{d_{2\ell-1}}\left(1+\frac{e_{2\ell-1}}{d_{2\ell}}\right),\quad{\rm since}\quad d_{\ell+1}>d_\ell,\\
& \le \frac{2}{d_{2\ell-1}} = \frac{4}{2\ell-1}\sqrt{\frac{\ell(2\ell-\frac{1}{2})(2\ell+\frac{1}{2})}{(\ell+1)(2\ell+1)^2}},\\
& \le \frac{4}{2\ell-1}\left(1 - \frac{32\ell^2+21\ell+4}{8(\ell+1)(2\ell+1)^2}\right) \le \frac{2}{\ell-\frac{1}{2}}.
\end{align}
For $c_\ell$, we start with the following observations
\begin{align}
\frac{e_\ell}{d_\ell} & = \frac{2}{\ell+1}\sqrt{\frac{(\ell+1)^2(\ell+\frac{1}{2})(\ell+\frac{3}{2})}{\ell(\ell+2)^3}},\\
& = \frac{2}{\ell+1}\sqrt{1-\frac{8\ell^3+25\ell^2+18\ell-3}{4\ell(\ell+2)^3}},\\
& \le \frac{2}{\ell+1}\left(1-\frac{8\ell^3+25\ell^2+18\ell-3}{8\ell(\ell+2)^3}\right) \le \frac{2}{\ell+1},
\end{align}
and
\begin{align}
\frac{f_\ell}{d_\ell} & = \sqrt{\frac{(\ell+1)^3(\ell+3)(2\ell+1)}{\ell(\ell+2)^3(2\ell+5)}},\\
& = \sqrt{1-\frac{2}{\ell}+\frac{10\ell^3+64\ell^2+128\ell+83}{\ell(\ell+2)^3(2\ell+5)}},\\
& \le 1-\frac{1}{\ell}+\frac{10\ell^3+64\ell^2+128\ell+83}{2\ell(\ell+2)^3(2\ell+5)},
\end{align}
and since $10\ell^3+64\ell^2+128\ell+83 \le 10\ell^3+65\ell^2+140\ell+100 = 5(\ell+2)^2(2\ell+5)$,
\begin{equation}
\frac{f_\ell}{d_\ell} \le 1-\frac{1}{\ell}+\frac{5}{2\ell(\ell+2)}.
\end{equation}
Thus,
\begin{align}
\norm{c_\ell}_\infty & = \max\left\{\frac{f_{2\ell-1}}{d_{2\ell-1}}+\frac{e_{2\ell-1}}{d_{2\ell-1}}\frac{f_{2\ell}}{d_{2\ell}} + \frac{e_{2\ell-1}}{d_{2\ell-1}}\frac{e_{2\ell}}{d_{2\ell}}, \frac{f_{2\ell}}{d_{2\ell}} + \frac{e_{2\ell}}{d_{2\ell}}\right\},\\
& \le \max\left\{\frac{f_{2\ell-1}}{d_{2\ell-1}}+\frac{1}{\ell}\frac{f_{2\ell}}{d_{2\ell}} + \frac{1}{\ell^2}, \frac{f_{2\ell}}{d_{2\ell}} + \frac{1}{\ell}\right\}.
\end{align}
The first term is
\begin{align}
\frac{f_{2\ell-1}}{d_{2\ell-1}}+\frac{1}{\ell}\frac{f_{2\ell}}{d_{2\ell}} + \frac{1}{\ell^2} & = 1 - \frac{1}{2\ell-1}+\frac{1}{\ell}+\frac{1}{2\ell^2}+\frac{5}{2}\left[\frac{1}{(2\ell-1)(2\ell+1)} + \frac{1}{2\ell^2(\ell+2)}\right],\\
& \le 1 + \frac{\frac{1}{2}}{\ell-\frac{1}{2}}+\frac{1}{2(\ell-\frac{1}{2})^2}+\frac{5}{2}\left[\frac{\frac{1}{4}}{(\ell-\frac{1}{2})^2} + \frac{\frac{1}{4}}{(\ell-\frac{1}{2})^2}\right],\\
& = 1 + \frac{\frac{1}{2}}{\ell-\frac{1}{2}}+\frac{\frac{7}{4}}{(\ell-\frac{1}{2})^2},
\end{align}
and the second term is smaller
\begin{align}
\frac{f_{2\ell}}{d_{2\ell}} + \frac{1}{\ell} & = 1+\frac{1}{2\ell}+\frac{5}{4\ell(2\ell+2)} \le 1+\frac{\frac{1}{2}}{\ell-\frac{1}{2}} + \frac{\frac{5}{8}}{(\ell-\frac{1}{2})^2}.
\end{align}
Thus, our rather crude upper bound is
\begin{equation}
\norm{c_\ell}_\infty \le 1 + \frac{\frac{1}{2}}{\ell-\frac{1}{2}}+\frac{\frac{7}{4}}{(\ell-\frac{1}{2})^2}.
\end{equation}
Together with $\log(1+x)\le x$ for $x\ge-1$, we find
\begin{align}
\prod_{\ell=i}^{j-1} \norm{c_\ell}_\infty & \le \prod_{\ell=i}^{j-1}\left(1 + \frac{\frac{1}{2}}{\ell-\frac{1}{2}}+\frac{\frac{7}{4}}{(\ell-\frac{1}{2})^2}\right),\\
& = \exp\left\{\sum_{\ell=i}^{j-1}\log\left[1 + \frac{\frac{1}{2}}{\ell-\frac{1}{2}}+\frac{\frac{7}{4}}{(\ell-\frac{1}{2})^2}\right]\right\},\\
& \le \exp\left\{\sum_{\ell=i}^{j-1}\left[\frac{\frac{1}{2}}{\ell-\frac{1}{2}}+\frac{\frac{7}{4}}{(\ell-\frac{1}{2})^2}\right]\right\},\\
& \le \exp\left\{\sum_{\ell=i}^j\frac{\frac{1}{2}}{\ell-\frac{1}{2}}+\sum_{\ell=1}^\infty\frac{\frac{7}{4}}{(\ell-\frac{1}{2})^2}\right\},\\
& = \exp\left\{\sum_{\ell=i}^j\frac{\frac{1}{2}}{\ell-\frac{1}{2}}+\frac{7\pi^2}{8}\right\},\\
& \le \exp\left\{\frac{7\pi^2}{8}+1+\int_i^j\frac{\frac{1}{2}}{x-\frac{1}{2}}\ud x\right\},\\
& = \exp\left\{\frac{7\pi^2}{8}+1+\log\sqrt{\frac{j-\frac{1}{2}}{i-\frac{1}{2}}}\right\} = e^{1+\frac{7\pi^2}{8}}\sqrt{\frac{j-\frac{1}{2}}{i-\frac{1}{2}}}.
\end{align}
Then, by the same arguments,
\begin{align}
\sum_{i=1}^j\prod_{\ell=i}^{j-1}\norm{c_\ell}_\infty^2 & \le e^{2+\frac{7\pi^2}{4}}(j-\tfrac{1}{2})\sum_{i=1}^j\frac{1}{i-\frac{1}{2}},\\
& \le e^{2+\frac{7\pi^2}{4}}(j-\tfrac{1}{2})\left[2+\log(2j-1)\right].
\end{align}
Finally,
\begin{align}
\norm{R^{-1}}_F^2 = 4\sum_{j=1}^n\norm{a_j^{-1}}_\infty^2\sum_{i=1}^j\prod_{\ell=i}^{j-1}\norm{c_\ell}_\infty^2 & \le \sum_{j=1}^n\frac{16}{(j-\frac{1}{2})^2}e^{2+\frac{7\pi^2}{4}}(j-\tfrac{1}{2})\left[2+\log(2j-1)\right],\\
& \le 16e^{2+\frac{7\pi^2}{4}}\left[2+\log(2n-1)\right]\sum_{j=1}^n\frac{1}{j-\frac{1}{2}},\\
& \le 16e^{2+\frac{7\pi^2}{4}}\left[2+\log(2n-1)\right]^2.
\end{align}

\begin{remark}
\begin{enumerate}
\item The constant $4e^{1+\frac{7\pi^2}{4}}$ in theorem~\ref{theorem:kappa} is a rather crude overestimate. Based on numerical evidence, we conjecture the following more accurate underestimate
\begin{equation}
\norm{R^{-1}}_2 \approx \frac{2}{\pi}\log(n+\tfrac{5}{2}),\quad{\rm for}\quad n>1.
\end{equation}
\item By equivalence of norms, we also have inequalities on the $1$-norm and $\infty$-norm condition numbers of $R$, though these do not directly translate to the least-squares problem defined by $M$.
\item Surprisingly, the condition number decreases as $\abs{m}\nearrow n$ even though the distinction between different harmonics at high order is less pronounced.
\end{enumerate}
\end{remark}

\end{document}